\documentclass[12pt]{article}

\usepackage{amsmath}
\usepackage{amssymb, latexsym, a4}
\usepackage{graphicx}
\usepackage[all]{xy}
\usepackage{multicol}
\usepackage[usenames]{color}
 
\usepackage{mathrsfs}
\usepackage{hyperref}
\usepackage{cleveref}
\newtheorem{theorem}{Theorem}[section]
\newtheorem{lemma}[theorem]{Lemma}
\newtheorem{remark}[theorem]{Remark}

\newtheorem{proposition}[theorem]{Proposition}

\newtheorem{definition}[theorem]{Definition}
\newtheorem{example}[theorem]{Example}
\newenvironment{proof}{\trivlist\item[]\rm{\textbf{Proof.}\ }}{\endtrivlist}
 \makeatletter
      \def\@setcopyright{}
      \def\serieslogo@{}
      \makeatother

\author{Hesam Safa and Morteza Norouzi \\ }

\title{Solvable Lie algebras derived from Lie hyperalgebras}

\begin{document}
\maketitle


\noindent\textbf{Abstract.} Recently in \cite{s-n}, we have investigated Lie algebras and abelian Lie algebras
derived from Lie hyperalgebras using the fundamental relations $\mathcal{L}$ and $\mathcal{A}$, respectively.
In the present paper, continuing this method we obtain solvable Lie algebras from Lie hyperalgebras by $\mathcal{S}_n$-relations. We
show that $\bigcap_{n\geq 1}\mathcal{S}^*_n$ is the smallest equivalence relation on a Lie hyperalgebra such that
the quotient structure is a solvable Lie algebra. We also provide some necessary and sufficient conditions for transitivity of the
relation $\mathcal{S}_n$ using the notion of $\mathcal{S}_n$-part. \\



\section{Introduction}
In the classical theory, the quotient of a group by a normal subgroup is a group.
 In 1934,  Marty \cite{ak} states that the quotient of a group by any subgroup is a {\it hypergroup}.
 Concretely, if $G$ is a group and $H$ is a subgroup of $G$, then the set of all left cosets of $H$ in $G$ is a hypergroup
 with the multiplication $$xH\cdot yH:=\{zH|\ z=xhy,\ h\in H\}.$$
  Clearly, if $H$ is normal in $G$, then the set of  cosets
 turns into the quotient group $G/H$.

 Following such an approach to algebraic structures, many investigations have been made
 on hyperstructure theory (hypergroups, hyperrings, hypermodules, hypervector spaces, hyperalgebras,  etc.)
 whose applications  nowadays have been known in other sciences such as algebra and geometry
 as well as automata, cryptography, artificial intelligence and probability, relational algebras,
 sensor networks, theoretical physics and chemistry (see \cite{cor,davv4,vo} for more information).
 An important part of  studies on hyperstructures is about strongly regular relations ({\it fundamental relations}) that
 actually makes a bridge between classical structures and hyperstructures. In 1970, this connection was achieved by  Koskas \cite{kos}
 using the relation $\beta$ on (semi)hypergroups and its transitive closure to obtain (semi)groups from a quotient of (semi)hypergroups.
 This relation has been studied mainly by Corsini \cite{cor1}, Vougiouklis \cite{vo}, Davvaz \cite{davv10}, Davvaz and Leoreanou-Fotea
 \cite{davv4}, Freni \cite{F01}, Migliorato \cite{mig} and many others. This quotient set not only links hyperstructures with classical
 structures, but also enhances the view of hyperstructures as a generalization of the corresponding classical algebraic
 structures. Hence, studies on this topic have been continued to obtain commutative semigroups \cite{F02}, cyclic groups \cite{mousavi},
 nilpotent groups \cite{davv1}, Engel groups \cite{engel}, solvable groups \cite{davv8}, rings \cite{voug2, davv4}, commutative rings
 \cite{davv6}, Boolean rings \cite{davv7}, modules over commutative rings \cite{davv3}, and also \cite{adineh1, adineh2, nor1, nor2}.

The authors in  \cite{s-n} studied  strongly regular relations on Lie hyperalgebras  and obtained Lie algebras and abelian Lie algebras  from Lie hyperalgebras using the fundamental relations $\mathcal{L}$ and $\mathcal{A}$, respectively. Now in this paper, we investigate solvable Lie algebras  derived from Lie hyperalgebras by $\mathcal{S}_n$-relations
($\mathcal{L}\subseteq\mathcal{S}_n\subseteq\mathcal{S}_1=\mathcal{A}$). We introduce
the relations $\mathcal{S}_n$  and show that their transitive closures are strongly regular relations on a Lie hyperalgebra $L$
such that the quotient $L/\mathcal{S}^*_n$ is a solvable Lie algebra of length at most
$n$. We also prove that if the fundamental Lie algebra $L/\mathcal{L}^*$ is  finite dimensional,
then $\mathcal{S}=\bigcap_{n\geq 1}\mathcal{S}^*_n$
is the smallest equivalence relation on  $L$ such that
$L/\mathcal{S}$ is a (fundamental) solvable Lie algebra. Moreover, it is shown that if $L$ is a simple Lie algebra, then
$L/\mathcal{S}=0$ and $L/\mathcal{L}\cong L$, which is in  coincidence with the concept of simplicity.
Finally, we give  some necessary and sufficient conditions for transitivity of the
relation $\mathcal{S}_n$ using its $\mathcal{S}_n$-part.

A hyperring is an algebraic hyperstructure $(R, +, \cdot)$ where $(R,+)$ is a hypergroup, $(R,\cdot)$ is a semihypergroup and the multiplication $"\cdot"$ is distributive with respect to the addition $"+"$. If $(R, \cdot)$ ($(R-\{0\}, \cdot)$, if $R$ contains 0) is a hypergroup, then $(R, +, \cdot)$ is said to be a hyperfield.
\begin{definition}\normalfont
Let $(R, +, \cdot)$ be a hyperring, $(M, +)$ be a hypergroup and $\mathscr{P}^{*}(M)$ denote the family of all non-empty subsets of $M$,
    together with an external map $R\times M\longrightarrow \mathscr{P}^{*}(M)$, defined by $(r,m)\mapsto r m$ such that for all
    $a, b\in R$ and $x, y\in M$, we have $a(x+y)=ax+ay$, $(a+b)x=ax+bx$ and $(ab)x=a(bx)$. Then $M$ is called a hypermodule over $R$. If we consider a hyperfield $F$ instead of a hyperring $R$, then $M$ is called a hypervector space.
\end{definition}

Throughout the paper, any hypervector space $M$ is assumed to have a scalar identity $0_{_M}$, and any hyperfield $F$
 contains a scalar identity $0_{_F}$ and a multiplication unit $1$ such that
$0_{_F} x=\{0_{_M}\}$  and $1 x=\{x\}$ for all $x\in M$.


\section{$\mathcal{L}$-relation and $\mathcal{A}$-relation on Lie hyperalgebras}

This section is devoted to discuss some preliminary results on Lie algebras obtained from Lie hyperalgebras from \cite{s-n}.

A Lie algebra $(L,\oplus,\odot)$  is a vector space over a field $\mathbb{F}$ equipped with a bilinear map $\lfloor -,- \rfloor : L \times L \to L$, usually called the
 Lie bracket of $L$,  satisfying the following conditions:
\begin{itemize}
\item[$(i)$] $\lfloor x,x\rfloor=0$, for all $x\in L$,
\item[$(ii)$] $\lfloor x,\lfloor y,z\rfloor\rfloor\oplus\lfloor y, \lfloor z,x\rfloor\rfloor\oplus\lfloor z,\lfloor x,y\rfloor\rfloor=0$, for all $x,y,z\in L$ (Jacobi identity).
\end{itemize}

\begin{definition}\label{def}\normalfont
Let $(L,+,\cdot)$  be a hypervector space   over a hyperfield
$(F,+,\cdot)$  and
$[-,-] : L \times L \to \mathscr{P}^{\ast}(L)$  given by $(x,y)\to[x,y]$ be a bilinear map
which we would call the {\it hyperbracket} of $L$.
Then $L$ is said to be a Lie hyperalgebra if the following axioms hold:
\begin{itemize}
\item[$(i)$] $0\in[x,x]$, for all $x\in L$,
\item[$(ii)$] $0\in \big{(}[x,[y,z]]+[y,[z,x]]+[z,[x,y]]\big{)}$, for all $x,y,z\in L$.
\end{itemize}
\end{definition}

We use the operations $"\oplus,\odot"$ and  $\lfloor-,-\rfloor$ as the bracket in Lie algebras, and
$"+, \cdot"$ and $[-,-]$
as the hyperbracket in Lie hyperalgebras.
Also, the bilinearity of the hyperbracket means:
\[[\lambda_1 x_1 +\lambda_2 x_2,y]=\lambda_1 [x_1,y]+\lambda_2[x_2,y],\]
\[[x,\lambda_1 y_1+\lambda_2 y_2]=\lambda_1[x,y_1]+\lambda_2[x,y_2],\]
for all $x,x_1,x_2,y,y_1,y_2\in L$ and $\lambda_1,\lambda_2\in F$ (see \cite{davv5,s-n} for details and examples).\\

Let $L$ be a Lie hyperalgebra on $F$ and $\rho\subseteq L\times L$ be an equivalence relation. For  non-empty subsets $A$
and $B$ of $L$, define
$$A\ \bar{\bar{\rho}}\ B \ \ \Longleftrightarrow \ \ a\ \rho\ b,\  \forall a\in A, \forall b\in B.$$
The relation $\rho$ is said to be strongly regular on the left (on the right) if $x\ \rho\ y $ implies $a+x\ \bar{\bar{\rho}}\ a+y$,
$\lambda\cdot x\ \bar{\bar{\rho}}\ \lambda\cdot y$ and $[a,x]\ \bar{\bar{\rho}}\ [a,y]$
(if $x\ \rho\ y$ implies $x+a\ \bar{\bar{\rho}}\ y+a$,
$x\cdot\lambda\ \bar{\bar{\rho}}\ y\cdot \lambda$ and $[x,a]\ \bar{\bar{\rho}}\ [y,a]$), for all $x,y,a\in L$ and $\lambda\in F$.
In addition, $\rho$ is called strongly regular if it is strongly
regular on the both left and right.

\begin{lemma}\label{lemma1}\cite[Proposition 2.5]{s-n}
Let $L$ be a Lie hyperalgebra over a hyperfield $F$, and $\rho$ and $\delta$ be equivalence relations on $L$ and  $F$, respectively.
 If $\rho$ and $\delta$ are strongly regular, then the quotient $L/\rho$ is a  Lie algebra over the field $F/\delta$, under the following operations:
 \begin{eqnarray*}
 \bar{x}\oplus\bar{y}&=&\bar{z}, \ \ \text{for all}\ z\in x+y\\
\widetilde{\lambda}\odot\bar{x}&=&\bar{z}, \ \ \text{for all}\ z\in \lambda\cdot x\\
\lfloor\bar{x},\bar{y}\rfloor&=&\bar{z}, \ \ \text{for all}\ z\in [x,y],
\end{eqnarray*}
for $\bar{x},\bar{y}\in L/\rho$ and $\widetilde{\lambda}\in F/\delta$.
 Conversely, if $L/\rho$ together with the above operations is a Lie algebra over the field $F/\delta$, then
  $\rho$ is strongly regular.
\end{lemma}

In \cite{s-n}, the relation $\mathcal{L}$ on a Lie hyperalgebra $L$ (over a hyperfield $F$) is defined as follows:
 \begin{eqnarray*}
x\ \mathcal{L}\ y \ \Longleftrightarrow  \{x,y\}\subseteq \sum_{i=1}^n \ell_i  &\text{where}& \ell_i=f_i(g_{i1},g_{i2},\ldots,g_{im_i})\\
   &\text{with}&  g_{ij}= \Big{(}\sum_{k=1}^{p_{ij}}\big{(} \prod_{r=1}^{q_{ijk}}\lambda_{ijkr} \big{)}\Big{)}\cdot h_{ij},
\end{eqnarray*}
such that $1\leq j\leq m_i$, $h_{ij}\in L$, $\lambda_{ijkr}\in F$,  and
$f_i$ is  an arbitrary {\it hyperbracket function}, i.e.
 $f_i(g_{i1},g_{i2},\ldots,g_{im_i})$ is  an arbitrary composition of the elements $g_{i1},g_{i2},\ldots,g_{im_i}$ only by  hyperbracket.
 Moreover, it is shown that
 $\mathcal{L}^{\ast}$ is a strongly regular relation on $L$, which is also the smallest equivalence relation on $L$ such that $L/\mathcal{L}^\ast$ is a (fundamental) Lie algebra (see Theorem 3.3 and Corollary 3.4 in \cite{s-n}).

 Furthermore, the  $\mathcal{A}$-relation on $L$ is defined as:
\begin{eqnarray*}
x\ \mathcal{A}\ y \ \Longleftrightarrow \ &\exists n\in\mathbb{N},\ \exists \sigma\in \mathbb{S}_n,\ \exists f_1,\ldots,f_n,\ \exists m_1,\ldots,m_n\in\mathbb{N},\
\exists\sigma_i\in\mathbb{S}_{m_i},\\
&\exists h_{i1},\ldots,h_{im_i}\in L,\ \exists p_{i1},\ldots,p_{im_{i}}\in\mathbb{N},\ \exists\sigma_{ij}\in\mathbb{S}_{p_{ij}},\
\exists q_{ij1},\ldots,q_{ijp_{ij}}\in\mathbb{N},\\
&\exists  \lambda_{ijk1},\ldots,\lambda_{ijkq_{ijk}}\in F,\ \exists\sigma_{ijk}\in\mathbb{S}_{q_{ijk}}
\end{eqnarray*}
such that
 \begin{eqnarray*}
 x\in \sum_{i=1}^n \ell_i \  \ \ \text{where}\ \ \ \ell_i=f_i(g_{i1},g_{i2},\ldots,g_{im_i})
 \ \ \ \text{with}\ \ \   g_{ij}= \Big{(}\sum_{k=1}^{p_{ij}}\big{(} \prod_{r=1}^{q_{ijk}}\lambda_{ijkr} \big{)}\Big{)}\cdot h_{ij},
\end{eqnarray*}
for $1\leq j\leq m_i$, and
 \begin{eqnarray*}
 y\in \sum_{i=1}^n \ell'_{i}  & \text{where}&  \ell'_i=f_{\sigma(i)}(g'_{i1},g'_{i2}, \ldots,g'_{im_{\sigma(i)}}) \  \ \text{with}\  \
  g'_{ij}=  A_{\sigma(i)\sigma_{\sigma(i)}(j)} \cdot h_{\sigma(i)\sigma_{\sigma(i)}(j)}\\
 &\text{in which}& 1\leq j\leq m_{\sigma(i)} \ \ \text{and}\ \ A_{ij}=\sum_{k=1}^{p_{ij}}B_{ij\sigma_{ij}(k)} \  \ \text{with}\  \ B_{ijk}=\prod_{r=1}^{q_{ijk}}
   \lambda_{ijk\sigma_{ijk}(r)}.
 \end{eqnarray*}
 It is proved that
  $\mathcal{A}^{\ast}$ is a strongly regular relation, and also the smallest equivalence relation on  $L$ such that
   $L/\mathcal{A}^\ast$ is an abelian Lie algebra (\cite[Theorem 4.2, Corollary 4.3]{s-n}).


\section{Main results}

In this section, we introduce and study the smallest equivalence  relation on a Lie hyperalgebra such that the quotient
structure is a solvable Lie algebra.
 A Lie algebra $L$ is called solvable of length $n$, if
$L^{(n)}=0$ and $L^{(n-1)}\not=0$, where $L^{(i)}$ denotes the $(i+1)$st term of the derived series of $L$, defined
inductively by $L^{(0)} =L$ and  $L^{(i)}=\lfloor L^{(i-1)},L^{(i-1)}\rfloor=\langle \lfloor x,y\rfloor|\ x,y\in L^{(i-1)}\rangle$ for $i\geq 1$.

Now, let $L$ be a Lie hyperalgebra. For non-empty subsets $A,B$ of $L$,   define
$$[A,B]=\bigcup_{a\in A,\ b\in B}[a,b],$$
and put $L^{[0]} =L$ and $L^{[i]}=[L^{[i-1]},L^{[i-1]}]$ for $i\geq 1$. Clearly, $L^{[i]}\subseteq L^{[i-1]}$ for every $i\geq 1$.
Also, for every Lie algebra $L$  which may be considered as a trivial Lie hyperalgebra (see Example \ref{ex1} below), we have
 $L^{[i]}\subseteq L^{(i)}$ for every $i\geq 0$.

\begin{definition}\label{def1}\normalfont
Let  $L$ be a Lie hyperalgebra over a hyperfield $F$.
For $n\in \mathbb{N}$, we define the  relation $\mathcal{S}_{n}$ on  $L$ as follows:
 \begin{eqnarray*}
x\ \mathcal{S}_{n}\ y \ \Longleftrightarrow \ &\exists t\in\mathbb{N},\  \exists \sigma\in \mathbb{S}_t,\ \exists f_1,\ldots,f_t,\ \exists
m_1,\ldots,m_t\in\mathbb{N},\ \exists h_{i1},\ldots,h_{im_i}\in L,\\
&\exists p_{i1},\ldots,p_{im_{i}}\in\mathbb{N},\ \text{and for all}\ \ 1\leq i\leq t \ \ \text{and}\ \ 1\leq j\leq m_i:\\
&\Big{(} \exists\sigma_i\in\mathbb{S}_{m_i}\ \ \text{such that}\ \ \sigma_i(j)=j \ \ \text{if}\ \ h_{ij}\not\in L^{[n-1]}\Big{)},
\ \text{and also} \\
&\exists\sigma_{ij}\in\mathbb{S}_{p_{ij}},\
\exists q_{ij1},\ldots,q_{ijp_{ij}}\in\mathbb{N},\ \exists \lambda_{ijk1},\ldots,\lambda_{ijkq_{ijk}}\in F,\ \exists\sigma_{ijk}\in\mathbb{S}_{q_{ijk}}
\end{eqnarray*}
such that
 \begin{eqnarray*}
 x\in \sum_{i=1}^t \ell_i  \ \ \ \text{where}\ \ \ \ell_i=f_i(g_{i1},g_{i2},\ldots,g_{im_i})
 \ \ \ \text{with}\ \ \  g_{ij}= \Big{(}\sum_{k=1}^{p_{ij}}\big{(} \prod_{r=1}^{q_{ijk}}\lambda_{ijkr} \big{)}\Big{)}\cdot h_{ij},
\end{eqnarray*}
for $1\leq j\leq m_i$, and
 \begin{eqnarray*}
 y\in \sum_{i=1}^t \ell'_{i}  & \text{where}&  \ell'_i=f_{\sigma(i)}(g'_{i1},g'_{i2}, \ldots,g'_{im_{\sigma(i)}}) \  \ \text{with}\  \
  g'_{ij}=  A_{\sigma(i)\sigma_{\sigma(i)}(j)} \cdot h_{\sigma(i)\sigma_{\sigma(i)}(j)}\\
 &\text{in which}& 1\leq j\leq m_{\sigma(i)} \ \ \text{and}\ \ A_{ij}=\sum_{k=1}^{p_{ij}}B_{ij\sigma_{ij}(k)} \  \ \text{with}\  \ B_{ijk}=\prod_{r=1}^{q_{ijk}}
   \lambda_{ijk\sigma_{ijk}(r)}.
 \end{eqnarray*}
\end{definition}

Clearly for every $n\geq 1$, $\mathcal{S}_n$ is a reflexive and symmetric relation and $\mathcal{L}\subseteq\mathcal{S}_{n+1}\subseteq\mathcal{S}_n\subseteq\mathcal{S}_1=\mathcal{A}$.
 Consider $\mathcal{S}_n^*$ as the transitive closure of $\mathcal{S}_n$ (\cite[Example 3.2]{s-n} shows that $\mathcal{S}_n$ is not
 necessarily  transitive, for all $n\geq 1$). Also,  $\mathcal{S}^*_n(x)$ denotes
 the equivalence class of $x\in L$.

\begin{proposition}\label{prop1}
For every $n\geq 1$, $\mathcal{S}_n^*$ is a strongly regular relation on a Lie hyperalgebra $L$.
\end{proposition}
\begin{proof}
We use an argument similar to the proof of \cite[Theorem 4.2]{s-n} on the $\mathcal{A}$-relation.
 We  show that $x\ \mathcal{S}_n \ y$ implies $[x,a]\ \bar{\bar{\mathcal{S}_n^*}}\ [y,a]$ for every $a\in L$.
Using above notations,  $x\in \sum_{i=1}^t \ell_i$ and $y\in\sum_{i=1}^t \ell'_i$, and hence
  \begin{eqnarray*}
&&[x,a]\subseteq\Big{[}\sum_{i=1}^t \ell_i,a\Big{]}=\sum_{i=1}^t \big{[}f_i(g_{i1},g_{i2},\ldots,g_{im_i}),a\big{]},\\
&&[y,a]\subseteq\Big{[}\sum_{i=1}^t \ell'_{i},a\Big{]}
=\sum_{i=1}^t\big{[}f_{\sigma(i)}(g'_{i1},g'_{i2}, \ldots,g'_{im_{\sigma(i)}}),a\big{]}.
 \end{eqnarray*}
 For all $1\leq i\leq t$,  consider  $\tau_i\in\mathbb{S}_{m_i +1}$ given by
$\tau_i(j)=\sigma_i(j)$ for all $1\leq j\leq m_i$ and $\tau_i(m_i+1)=m_i+1$, and put  $g_{i(m_i +1)}=h_{i(m_i +1)}=a$
and  $f_i^*(-,a)=[f_i(-),a]$.
 Thus
 \begin{eqnarray*}
 [x,a]\subseteq \sum_{i=1}^{t} f_i^*(g_{i1},g_{i2},\ldots,g_{im_i},g_{i(m_i +1)}) \ \
 \text{with}\ \    g_{ij}= \Big{(}\sum_{k=1}^{p_{ij}}\big{(} \prod_{r=1}^{q_{ijk}}\lambda_{ijkr} \big{)}\Big{)}\cdot h_{ij}
\end{eqnarray*}
for $1\leq j\leq m_i +1$, and
 \begin{eqnarray*}
 [y,a]\subseteq   \sum_{i=1}^{t} f_{\sigma(i)}^*(g'_{i1},g'_{i2},\ldots,g'_{im_{\sigma(i)}},g'_{i(m_{\sigma(i)} +1)})
  \ \ \text{with}\ \  g'_{ij}= A_{\sigma(i)\tau_{\sigma(i)}(j)} \cdot h_{\sigma(i)\tau_{\sigma(i)}(j)}\\
 \text{in which} \ \ 1\leq j\leq m_{\sigma(i)} +1 \ \ \text{and}\ \ A_{ij}= \sum_{k=1}^{p_{ij}}B_{ij\sigma_{ij}(k)}  \ \ \text{with}\ \ B_{ijk}=\prod_{r=1}^{q_{ijk}}
   \lambda_{ijk\sigma_{ijk}(r)},
\end{eqnarray*}
such that $\tau_i(j)=j$ if $h_{ij}\not\in L^{[n-1]}$, for $1\leq i\leq t$ and $1\leq j\leq m_i +1$.
Note that due to the definition of $\tau_i$, it does not matter whether the element
$h_{i(m_i +1)}=a$ belongs to $L^{[n-1]}$ or not.
It follows that $u\ \mathcal{S}_n^*\ v$ for all $u\in [x,a]$ and  $v\in [y,a]$, and so
 $[x,a]\ \bar{\bar{\mathcal{S}_n^*}}\ [y,a]$.
Similarly, we may prove that $[a,x]\ \bar{\bar{\mathcal{S}_n^*}}\ [a,y]$.
Also similar to  \cite[Lemma 2.2]{davv3},  one can  show that $\mathcal{S}_n^*$ is strongly regular  on $(L, +, \cdot)$,
which  completes the proof. $\ \  \Box$
\end{proof}

Let $(R,+,\cdot)$ be a hyperring and $x,y\in R$. In \cite{davv6}, the relation $\alpha$ is defined on $R$ as follows:
 \begin{eqnarray*}
x\ \alpha\ y \ \Longleftrightarrow \ \exists t\in\mathbb{N},\  \exists \sigma\in \mathbb{S}_t,\ \exists k_1,\ldots,k_t\in\mathbb{N},\
\exists\sigma_i\in\mathbb{S}_{k_i},\
\exists x_{i1},\ldots,x_{ik_i}\in R
\end{eqnarray*}
such that
 \begin{eqnarray*}
x\in \sum_{i=1}^t\Big{(} \prod_{j=1}^{k_i} x_{ij} \Big{)}\ \ \text{and}\ \ y\in\sum_{i=1}^t A_{\sigma(i)}\ \ \text{where}\ \
A_i=\prod_{j=1}^{k_i}x_{i\sigma_{i}(j)}.
\end{eqnarray*}
Also, $\alpha^\ast$ is the transitive closure of $\alpha$ (see also \cite[Definition 7.1.1]{davv4}).

A Lie algebra $L$, over a field  $\mathbb{F}$, is said to be abelian, if $L^{(1)}=0$.
It is easy to see that if the characteristic of $\mathbb{F}$ is not 2 and $\lfloor x,y\rfloor=\lfloor y,x\rfloor$ for all $x,y\in L$, then  $L$ is abelian (see also \cite{s-n}). In the following results,  $F$ is a hyperfield such that the characteristic of the
field $F/\alpha^\ast$ is not $2$.

\begin{theorem}\label{th2}
Let $L$ be a Lie hyperalgebra over  $F$.  Then
$L/\mathcal{S}_n^*$ is a solvable Lie algebra of length at most $n$ over  $F/\alpha^\ast$.
\end{theorem}
\begin{proof}
By Lemma \ref{lemma1} and Proposition \ref{prop1}, $L/\mathcal{S}_n^*$ is a Lie algebra over $F/\alpha^\ast$ (note that
$(L/\mathcal{S}_n^*,\oplus)$ is
also an abelian group  (see \cite{davv3,s-n})). We show that $L/\mathcal{S}_n^*$ is solvable of length at most $n$.
One may inductively prove that
\begin{equation}\label{eq1}
(L/\mathcal{S}_n^*)^{(i)}=\langle\mathcal{S}_n^*(h)|\ h\in L^{[i]}\rangle,
 \end{equation}
 for all $i\geq 0$.
In Definition \ref{def1}, let $t=1$, $m_1=2$, $g_{11}=h_{11}\in L^{[n-1]}$, $g_{12}=h_{12}\in L^{[n-1]}$, $f_1(g_{11},g_{12})=[g_{11},g_{12}]$,  and
  $\sigma_1\in \mathbb{S}_2$ such that $\sigma_1(1)=2$ and $\sigma_1(2)=1$.
  For every $x\in[g_{11},g_{12}]$ and $y\in[g_{1\sigma_1(1)},g_{1\sigma_1(2)}]=[g_{12},g_{11}]$, we have
 $x\ \mathcal{S}_n\ y$. Hence $x\ \mathcal{S}_n^*\ y$ and thus
 \begin{equation}\label{eq2}
\lfloor \mathcal{S}_n^*(g_{11}),\mathcal{S}_n^*(g_{12})\rfloor=\mathcal{S}_n^*(x)=\mathcal{S}_n^*(y)=
 \lfloor\mathcal{S}_n^*(g_{12}),\mathcal{S}_n^*(g_{11})\rfloor.
 \end{equation}
  Now since $g_{11},g_{12}\in L^{[n-1]}$,  equalities (\ref{eq1}) and (\ref{eq2}) imply that
  $(L/\mathcal{S}_n^*)^{(n-1)}$ is an abelian Lie algebra, which means $(L/\mathcal{S}_n^*)^{(n)}=0_{_{L/\mathcal{S}_n^*}}$.
  Therefore, $L/\mathcal{S}_n^*$ is solvable of length at most $n$.  $\ \  \Box$
\end{proof}

The following  trivial Lie hyperalgebra illustrates  the difference between the relations and also the quotients of a Lie hyperalgebra
by them.
\begin{example}\label{ex1} \normalfont
Let $(L,\oplus,\odot)$ be a 4-dimensional vector space  over $\mathbb{R}$ with a basis $\{a,b,c,d\}$. It is easy to check that $L$
is a   Lie hyperalgebra (over the trivial hyperfield $\mathbb{R}$)
  with the  hyperoperations:
$x+ y :=\{x\oplus y\}$,  $r\cdot x :=\{r\odot x\}$ for all $x,y\in L$ and $r\in \mathbb{R}$,
and the  bilinear hyperbracket:
\[\begin{array}{c|cccc}
[- ,- ] & a & b & c & d  \\ \hline
a & \{0\} & \{0\} & \{0\} & \{0\} \\
b & \{0\} & \{0\} & \{a\} & \{b\} \\
c & \{0\} & \{-1\odot a\} & \{0\}& \{-1\odot c\} \\
d & \{0\} & \{-1\odot b\} & \{c\} & \{0\}
\end{array}\]
Clearly, $L/\mathcal{L}^*$ is the 4-dimensional Lie algebra (over $\mathbb{R}/\gamma^*\cong\mathbb{R}$) with the basis $\{\bar{a},\bar{b},\bar{c},\bar{d}\}$
and  brackets $\lfloor \bar{b},\bar{c}\rfloor=\bar{a}$, $\lfloor \bar{b},\bar{d}\rfloor=\bar{b}$ and $\lfloor \bar{d},\bar{c}\rfloor=\bar{c}$
 where $\bar{x}=\mathcal{L}^*(x)$ for every $x\in L$ (see \cite{s-n,voug2} for
 the definition of the $\gamma$-relation).

Also,    $a\in[b,c]$ and $(-1\odot a)\in [c,b]$, which imply that $a\ \mathcal{A}\ (-1\odot a)$ and hence
$\widetilde{a}=\widetilde{{-1\odot a}}$, or equivalently $\widetilde{a}=\widetilde{0}$  where $\widetilde{x}=\mathcal{A}^*(x)$.
By using a similar manner, we get  $\widetilde{b}=\widetilde{c}=\widetilde{0}$. Therefore,
 $L/\mathcal{A}^*$ is the abelian 1-dimensional Lie algebra (over $\mathbb{R}/\alpha^*\cong\mathbb{R}$) with the basis
 $\{\widetilde{d}\}$.

 Moreover,  $b\in[b,d]$ and $c\in[d,c]$ which means $b,c\in L^{[1]}$. Now since  $a\in[b,c]$ and $(-1\odot a)\in [c,b]$, we get $a\ \mathcal{S}_2\ (-1\odot a)$ and hence $\widehat{a}=\widehat{0}$ where $\widehat{x}=\mathcal{S}_2^*(x)$. Observe  that $d\not\in L^{[1]}$, so
 $\widehat{b},\widehat{c}$ and $\widehat{d}$
 are non-zero.
 Thus, $L/\mathcal{S}_2^*$ is the 3-dimensional Lie algebra over $\mathbb{R}$ with the basis
 $\{\widehat{b},\widehat{c},\widehat{d}\}$ and  brackets $\lfloor\widehat{b},\widehat{d}\rfloor=\widehat{b}$ and
 $\lfloor\widehat{d},\widehat{c}\rfloor=\widehat{c}$. It is easy to see that $L/\mathcal{S}_2^*$ is solvable of length 2.

 Finally, it is not difficult to show that  $\mathcal{S}_n=\mathcal{L}$ for all $n\geq 3$ (note that $L/\mathcal{L}^*$ is solvable of length 3). In this example, we have
 $\mathcal{L}\subsetneqq\mathcal{S}_2\subsetneqq\mathcal{A}$ and
  $\dim(L/\mathcal{A}^*)\lneqq\dim(L/\mathcal{S}_2^*)\lneqq\dim(L/\mathcal{L}^*)$.  \\
\end{example}

Let $L$ be a  Lie hyperalgebra  such that the fundamental Lie algebra
$L/\mathcal{L}^*$ is  finite dimensional,  and define $$\mathcal{S}=\bigcap_{n\geq 1}\mathcal{S}^*_n.$$ Since
\[\dim\frac{L}{\mathcal{A}^*}\leq \dim\frac{L}{\mathcal{S}^*_2}\leq \cdots\leq\dim\frac{L}{\mathcal{S}^*_{n}}\leq\dim\frac{L}{\mathcal{S}^*_{n+1}}\leq
\cdots\leq\dim\frac{L}{\mathcal{L}^*},\]
for all $n\geq 2$,  there exists $m\in \mathbb{N}$ such that
 $\dim(L/\mathcal{S}^*_{m})=\dim(L/\mathcal{S}^*_{m+i})$ for all $i\geq 1$. Thus,
 $L/\mathcal{S}^*_{m}\cong L/\mathcal{S}^*_{m+i}$ (as vector spaces) and
  $\mathcal{S}^*_{m}=\mathcal{S}^*_{m+i}$ for all $i\geq 1$, which implies that $\mathcal{S}=\mathcal{S}^*_{m}$
  (for instance, in the above example $m=3$).
  Therefore by Proposition \ref{prop1} and Theorem \ref{th2}, $\mathcal{S}$ is a strongly regular relation on
   $L$ and $L/\mathcal{S}$ is a solvable Lie algebra.

\begin{theorem}
Let $L$ be a  Lie hyperalgebra such that $L/\mathcal{L}^*$ is  finite dimensional. Then $\mathcal{S}$ is the smallest equivalence relation
on $L$ such that $L/\mathcal{S}$ is a  solvable Lie algebra.
\end{theorem}
\begin{proof}
Let $\xi$ be an equivalence relation on $L$ such that
 $L/\xi$ is a solvable Lie algebra of length $n$ say, and also $(L/\xi,\oplus)$ is an abelian group.
   It is easy to see that if $h\in L^{[n]}$, then $\xi(h)\in (L/\xi)^{(n)}=0_{L/\xi}$.
 Now if $x\ \mathcal{S}_{n+1}\ y$, then using the notations of Definition \ref{def1}, we have
  $x\in \sum_{i=1}^t \ell_i$ and $y\in \sum_{i=1}^t \ell_{\sigma(i)}$. Since $(L/\xi,\oplus)$ is  abelian,
we get $\xi(x)=\oplus_{i\in I}\xi(\ell_i)=\xi(y)$ where $I\subseteq\{1,\ldots,t\}$ such that
  none of the elements  $h_{i1},h_{i2},\ldots,h_{im_i}$ belong to $L^{[n]}$ for every $i\in I$. Therefore, $x\ \xi\ y$ and so $\mathcal{S}^*_{n+1}\subseteq\xi$
which implies that $\mathcal{S}\subseteq\xi$. This completes the proof.   $\ \  \Box$
\end{proof}

A non-abelian Lie algebra $L$ is called {\it simple}, if it contains  no ideals other than 0 and $L$.
Also, a Lie algebra $L$ is said to be  {\it perfect}, if $L=L^{(1)}$.
It is easy to see that every simple Lie algebra is perfect (see \cite{erd}).

\begin{proposition}
Let $L$ be a perfect Lie algebra. Then $L/\mathcal{S}=0$, $L/\mathcal{L}\cong L$  and
$$\mathcal{L}=\{(x,x)|\ x\in L\}\subseteq \mathcal{S}=L\times L.$$
\end{proposition}
\begin{proof}
Clearly, for every Lie algebra  (trivial Lie hyperalgebra) $L$, we have $\mathcal{L}=\mathcal{L}^*=\{(x,x)|\ x\in L\}$
 (the diagonal relation on $L$) and $L/\mathcal{L}\cong L$.
Now, let $a\in L$. Since $L$ is perfect, we have $L=L^{(n)}$ for all $n\geq 1$. Without loss of generality, one may suppose that
$a=\lfloor x,y\rfloor$ where $x,y\in L^{[n-1]}$. Hence $a\in [x,y]$ and $(-1\odot a)=\lfloor y,x\rfloor\in[y,x]$. Thus
 $a\ \mathcal{S}_n\ (-1\odot a)$
and so  $\bar{a}=\bar{0}$ where $\bar{a}=\mathcal{S}^*_n(a)$.  Therefore $L/\mathcal{S}^*_n=0$ and
$\mathcal{S}^*_n=L\times L$ for all $n\geq 1$, which completes the proof.  $\ \  \Box$
\end{proof}

\begin{example}\label{ex2} \normalfont
Let $L$ be the 3-dimensional Lie algebra with the basis $\{a,b,c\}$ and non-zero Lie brackets
$\lfloor a,b\rfloor=c$, $\lfloor b,c\rfloor=a$ and $\lfloor c,a\rfloor=b$.
Clearly, $L$ is simple  (perfect) and $\mathcal{S}_n=\mathcal{A}$ for all $n\geq 1$.
One can easily check that $L/\mathcal{S}=0$ and $L/\mathcal{L}\cong L$.
\end{example}

\begin{remark}\normalfont
Using a similar argument as in Example \ref{ex1}, one may show that for a finite dimensional Lie algebra $L$, we have
 $$\dim(L/\mathcal{S}^*_n)=\dim L-\dim L^{(n)}.$$
Now, let  $m$ be the smallest positive integer such that
$L^{(m)}=L^{(m+i)}$ for all $i\geq 1$. Then $\dim(L/\mathcal{S})=\dim L-\dim L^{(m)}$.
In particular if $L$ is  solvable, then   $L/\mathcal{S}\cong L$ (in this case $\mathcal{S}=\mathcal{L}$).
\end{remark}


In what follows, we discuss the transitivity  conditions  of $\mathcal{S}_n$ (for a fixed $n\geq 1$).
 Under the  notations of Definition \ref{def1}, a non-empty subset $K$  of a Lie hyperalgebra $L$ (over an arbitrary
  hyperfield $F$) is said to be an
 $\mathcal{S}_n$-part of $L$ if for every $t\in \mathbb{N}$, every $\sigma\in \mathbb{S}_t$,
   every $\ell_1,\ldots,\ell_t$
  such that  $\ell_i=f_i(g_{i1},g_{i2},\ldots,g_{im_i})$
 with   $g_{ij}= \Big{(}\sum_{k=1}^{p_{ij}}\big{(} \prod_{r=1}^{q_{ijk}}\lambda_{ijkr} \big{)}\Big{)}\cdot h_{ij}$, and
  every $\sigma_i\in \mathbb{S}_{m_i}$ such that $\sigma_i(j)=j$ if $h_{ij}\not\in  L^{[n-1]}$, we have
\[\sum_{i=1}^t \ell_i \cap K\not=\emptyset \ \ \ \Longrightarrow\ \ \ \ \sum_{i=1}^t \ell'_{i}\subseteq K,\]
where $\ell'_i$ is that given in Definition \ref{def1}.

\begin{example}\normalfont
In Example \ref{ex1}, it is easy to see that the vector subspace $\langle a\rangle$ is an $\mathcal{S}_n$-part of $L$
for all $n\geq 1$. Moreover, the singleton $\{a\}$ is an $\mathcal{S}_n$-part ($\mathcal{L}$-part) of $L$
only for  $n\geq 3$, since $[b,c]\cap \{a\}\not=\emptyset$ and $b,c\in L^{[1]}$ but $[c,b]\nsubseteq\{a\}$.
Also, $\{b\}$ is an $\mathcal{S}_n$-part of $L$ for all $n\geq 2$, since $d\not\in L^{[1]}$.
\end{example}

\begin{lemma}\label{lem31}
Let $K$ be a non-empty subset of a Lie hyperalgebra $L$ and $x,y\in L$. Then $(i)$ $K$ is an $\mathcal{S}_n$-part of $L$, if and only if $(ii)$ $x\in K$ and $x\ \mathcal{S}_n\ y$ implies $y\in K$, if and only if $(iii)$ $x\in K$ and $x\ \mathcal{S}^*_n\ y$ implies $y\in K$.
\end{lemma}
\begin{proof}
$(i)\Rightarrow(ii)$ Let $x\in K$ and $x\ \mathcal{S}_n\ y$. Then  there exist $\ell_1,\ldots,\ell_t$
 such that $x\in \sum_{i=1}^t \ell_i$ and $y\in \sum_{i=1}^t \ell'_i$.
Thus $x\in \sum_{i=1}^n \ell_i\cap K$ and by $(i)$ we obtain $y\in K$.

$(ii)\Rightarrow(iii)$ Let $x\in K$ and $x\ \mathcal{S}^*_n\ y$. Then there exist $m\in \mathbb{N}$ and $x=z_1,z_2\ldots,z_{m-1},z_m=y\in L$
such that $x=z_1\ \mathcal{S}_n\ z_2\ \mathcal{S}_n \ldots \mathcal{S}_n\ z_{m-1}\  \mathcal{S}_n\ z_m=y$. Since $x\in K$, by $(ii)$ we have $z_2\in K$.
Repeating  this process ($m-1$ times) yields $y\in K$.

$(iii)\Rightarrow(i)$ Assume that $\sum_{i=1}^t \ell_i \cap K\not=\emptyset$. Then there exists $x\in \sum_{i=1}^t \ell_i \cap K$.
Suppose that   $y\in \sum_{i=1}^t \ell'_i$. Thus   $x\ \mathcal{S}^*_n\ y$ and by $(iii)$, $y\in K$ which implies that
  $\sum_{i=1}^t \ell'_{i}\subseteq K$.   $\ \  \Box$
\end{proof}

Under the  notations of Definition \ref{def1}, consider the sets
$T_{t}(x)=\Big{\{}(\ell_1,\ldots,\ell_t)|\ x\in\sum_{i=1}^t \ell_i\Big{\}}$
such that $\ell_i =f_i(g_{i1},g_{i2},\ldots,g_{im_i})$ with $g_{ij}=\Big{(}\sum_{k=1}^{p_{ij}}\big{(} \prod_{r=1}^{q_{ijk}}\lambda_{ijkr} \big{)}\Big{)}\cdot h_{ij}$, \ \ $P_t(x)=\bigcup_{t\geq 1}\Big{\{} \sum_{i=1}^t \ell'_{i}|\ (\ell_1,\ldots,\ell_t)\in T_t(x)\Big{\}}$ and
$P(x)=\bigcup_{t\geq 1} P_t(x)$.

\begin{remark} \label{rem1}\normalfont
It is easy to check that  $P(x)=\{y\in L|\ x\ \mathcal{S}_n\ y\}$ for every $x\in L$.
\end{remark}

In the following theorem, we give  the transitivity conditions of  $\mathcal{S}_n$.
\begin{theorem}
Let  $L$ be a Lie hyperalgebra over a hyperfield $F$. Then $(i)$ $\mathcal{S}_n$ is transitive, if and only if $(ii)$ $\mathcal{S}^*_n(x)=P(x)$ for every $x\in L$, if and only if $(iii)$ $P(x)$ is an $\mathcal{S}_n$-part of $L$ for every $x\in L$.
\end{theorem}
\begin{proof}
$(i)\Rightarrow(ii)$ By Remark \ref{rem1} and $(i)$, we trivially have $\mathcal{S}^*_n(x)=P(x)$.

$(ii)\Rightarrow(iii)$  Clearly $x\in P(x)$. If $x\ \mathcal{S}^*_n\ y$, then
$y\in \mathcal{S}^*_n(x)$ and by $(ii)$, we get $y\in P(x)$. Now in Lemma \ref{lem31} $(iii)\Rightarrow(i)$, put
 $K=P(x)$. It  implies that $P(x)$ is an $\mathcal{S}_n$-part of $L$.

$(iii)\Rightarrow(i)$ Let $x\ \mathcal{S}_n\ y$ and $y\ \mathcal{S}_n\ z$.
By Remark \ref{rem1}, we have $y\in P(x)$ and since $y\ \mathcal{S}_n\ z$, there exist
$t\in\mathbb{N}$ and $\ell_1,\ldots,\ell_{t}$  such that
$y\in\sum_{i=1}^{t} \ell_i\cap P(x)$ and $z\in \sum_{i=1}^{t} \ell'_{i}$. Now by $(iii)$, we have
$\sum_{i=1}^{t} \ell'_{i}\subseteq P(x)$ and hence $z\in P(x)$. Then by Remark \ref{rem1}, we get   $x\ \mathcal{S}_n\ z$
which shows that $\mathcal{S}_n$ is transitive. $\ \  \Box$
\end{proof}


{\bf Hesam Safa} \\
Department of Mathematics, Faculty of Basic Sciences, University of Bojnord, Bojnord, Iran.\\
E-mail address:  h.safa@ub.ac.ir \\
ORCID: 0000-0002-5418-8104  \\

{\bf Morteza Norouzi} \\
Department of Mathematics, Faculty of Basic Sciences, University of Bojnord, Bojnord, Iran.\\
E-mail address:  m.norouzi@ub.ac.ir,  \ \ \   m.norouzi65@yahoo.com \\
ORCID: 0000-0001-9850-1126  \\

\end{document}